\documentclass[10pt]{article}
\title{Heavy Ball and Nesterov Accelerations with Hessian-driven Damping for Nonconvex Optimization}
\author{N. Hadjisavvas\thanks{Department of Product and Systems Design Engineering, University of the Aegean, Hermoupolis, Syros, Greece.  E-mail: nhad@aegean.gr, ORCID-ID: 0000-0002-9895-8190} 
\and
F. Lara\thanks{Instituto de Alta investigaci\'on (IAI), Universidad de Tarapac\'a, Arica, Chile. E-mail: felipelaraobreque@gmail.com; flarao@academicos.uta.cl. Web: felipelara.cl, ORCID-ID: 0000-0002-9965-0921} 
\and
R.T.  Marcavillaca \thanks{Centro de Modelamiento Matem\'atico (CNRS UMI2807), Universidad de Chile, Santiago, Chile. E-mail: raultm.rt@gmail.com;
rtintaya@cmm.uchile.cl, ORCID-ID: 0000-0003-3748-0768}
\and
P.T. Vuong\thanks{School of Mathematical Sciences, University of Southampton, SO17 1BJ, Southampton, United Kingdom. E-mail: t.v.phan@soton.ac.uk, ORCID-ID: 0000-0002-1474-99}}

\usepackage{amsmath}
\usepackage{amsthm}
\usepackage{amssymb}
\usepackage{multicol}
\usepackage[active]{srcltx}
\usepackage{algorithm}
\usepackage{array}
\usepackage[dvipsnames,svgnames]{xcolor}
\usepackage{tikz}
\usepackage{subfig}
\usepackage{graphicx}
\usepackage{pgfplots}
\usepackage{xcolor}
\usepackage{amsmath}
\usepackage{amsthm}
\usepackage{amssymb}
\usepackage{comment}
\usepackage{multicol}
\usepackage[active]{srcltx}
\usepackage{algorithm}
\usepackage{array}
\usepackage{algorithm, algorithmic}
\usepackage{ulem}  
\providecommand{\U}[1]{\protect \rule{.1in}{.1in}}
\newtheorem{theorem}{Theorem}

\newtheorem{corollary}[theorem]{Corollary}

\newtheorem{example}[theorem]{Example}

\newtheorem{lemma}[theorem]{Lemma}

\newtheorem{proposition}[theorem]{Proposition}
\newtheorem{remark}[theorem]{Remark}

\begin{document}

\maketitle

\begin{abstract}
\noindent In this work, we investigate a second-order dynamical system with Hessian-driven damping tailored for a class of nonconvex functions called strongly quasiconvex. Buil\-ding upon this continuous-time model, we derive two discrete-time gra\-dient-based algorithms through time discretizations. The first is a Heavy Ball method with Hessian correction, incorporating cur\-va\-tu\-re-dependent terms that arise from discretizing the Hessian damping component. The second is a Nesterov-type accelerated method with adaptive momentum, fea\-tu\-ring correction terms that account for local curvature. Both algorithms aim to enhance stability and convergence performance, particularly by mi\-ti\-ga\-ting oscillations commonly observed in cla\-ssi\-cal momentum me\-thods. Furthermore, in both cases we establish li\-near convergence to the optimal solution for the iterates and functions values. Our approach highlights the rich interplay between continuous-time dynamics and discrete optimization algorithms in the se\-tting of strongly quasiconvex objectives. Numerical experiments are presented to support 
obtained results.

\medskip

\noindent{\small \emph{Keywords}: Nonconvex optimization; Dynamical systems; Hessian driven damping; Heavy-ball acceleration; Nesterov acceleration; Linear convergence.}
\end{abstract}

\section{Introduction}

Let $h: \mathbb{R}^n \rightarrow \mathbb{R}$ be a differentiable function. We are interested in the study of the minimization problem  
\begin{equation}\label{min.h}
 \min_{x \in \mathbb{R}^n} h(x),
\end{equation}
when the function $h$ belongs to a class of quasiconvex functions known as \textit{strongly quasiconvex functions} \cite{P}.

Quasiconvex functions play an important role in several areas of applied mathematics such as mathematical programming, mathematical finance, mi\-ni\-max theory and game theory \cite{ADSZ,CM-Book,HKS} as well as in economic analysis, particularly in production and utility theory \cite{AE,D-1959}. The significance of quasiconvexity lies in the natural {\it tendency to diversification} assumption on the consumers (see \cite{D-1959}) as well as in their ability to generalize convexity while preserving convex sublevel sets, ma\-king them especially suitable for modeling problems in other fields of mathematical applications which involve nonconvex geometries.

Strongly quasiconvex functions, introduced by Polyak in his seminal work \cite{P}, belong to a class of quasiconvex functions that satisfies several key pro\-per\-ties. These include the existence and uniqueness of global minimizers with a quadratic growth condition, gradient dominance, and bounded nonconvexity. Moreover, well-known algorithms such as proximal point and gradient based methods exhi\-bit linear convergence to the global optimal solution for such functions (see \cite{GLM-1,ILMY,Lara-9,LMV,LV}). Strongly quasiconvex functions include all strongly convex functions and also relevant nonconvex functions such as the square root of the Euclidean norm \cite{Lara-9}, and certain function ratios in which a strongly convex numerator is divided by a positive concave denominator \cite{ILMY}. Applications of strongly quasiconvex functions arise in diverse areas such as the calculus of variations, where they model energy functionals with controlled nonconvexity \cite{KOS}, Bayesian inference \cite{TIW} and  nonconvex minimax optimization \cite{CGC} among others. For a comprehensive overview of the theory and recent developments in this field, we refer the reader to the survey \cite{GLM-Survey}.

Gradient descent is a computationally efficient method for minimizing di\-ffe\-ren\-tiable functions. Starting from a point \( x_0 \in \mathbb{R}^n \), it follows the iteration:
\begin{equation}\label{eq:gradient}
 x_{k+1} = x_k - \beta_k \nabla h(x_k), ~~ \forall ~ k \geq 0,
\end{equation}
where \( \{\beta_k\}_{k} \) denotes a sequence of step sizes. This method can be viewed as a discretization of the first-order differential equation 
$$\dot{x}(t) = -\nabla h(x(t)).$$ 

Despite its simplicity, gradient descent often suffers from slow convergence and oscillatory (``zigzag'') behavior, particularly in ill-conditioned or nonconvex problems. To address these drawbacks, Polyak proposed an accelerated gradient scheme derived from the second-order dynamical system now known as the \textit{Heavy Ball with Friction} (HBF) system \cite{P2}:
\begin{equation}\label{eq:polyak_ode}
\ddot{x}(t) + \alpha \dot{x}(t) + \nabla h(x(t)) = 0, \quad t > 0, \quad \alpha > 0,
\end{equation}
where the damping term \( \alpha \dot{x}(t) \) reduces oscillations and the inertial term \( \ddot{x}(t) \) acce\-le\-ra\-tes convergence. Its discrete-time counterpart is precisely the well-known Heavy Ball Method given by
\begin{equation}\label{eq:heavy_ball}
 x^{k+1} = x^k + \alpha_k (x^k - x^{k-1}) - \beta_k \nabla h(x^k),
\end{equation}
where $\{\alpha_k\}_{k}$ is a momentum parameter sequence controlling the influence of previous iterates.

The Heavy Ball Method has been widely studied in both convex and nonconvex settings (see, e.g.,~\cite{Alv,Att,AC2017,ADR2,ADR,GOM,LMV}). Under suitable conditions, HBF can accelerate convergence relative to gradient descent, especially in problems with high curvature or in the presence of saddle points and nonconvexity, common situations in machine learning \cite{LVLL,MJ} and its applications.

Recently, the study of optimization methods via discretization of dynamical systems has attracted considerable interest. In convex settings, such methods can achieve improved convergence rates and reduced oscillations. In nonconvex settings, HBF may still outperform basic gradient descent, even though it may exhibit oscillatory behavior in poorly conditioned regions. These insights have motivated the development of more refined acceleration techniques, including Nesterov’s method \cite{Y-1983}, which offers faster convergence in the convex case \cite{SBC}.

For strongly convex and strongly quasiconvex functions, the HBF method is known to yield exponential convergence of \( h(x(t)) \) to the minimum value $\min_{\mathbb{R}^{n}}\,h$, but for general convex functions, the convergence rate of the HBF is limited to \(\mathcal{O}(1/t) \) in the worst case, matching that of gradient descent. Moreover, oscillations may persist when the function is ill-conditioned or has flat regions.

For dealing with these situations, the authors in \cite{AABR} proposed an enhancement of HBF through a Hessian-driven damping term. This leads to the modified system:
\begin{equation}\label{eq:hessian_damped}
\ddot{x}(t) + \alpha \dot{x}(t) + \beta \nabla^2 h(x(t)) \dot{x}(t) + \nabla h(x(t)) = 0, \quad t > 0,
\end{equation}
where the additional term \(\nabla^2 h(x(t)) \dot{x}(t)\) adapts the damping effect according to the local curvature of the function. This modification effectively suppresses oscillations that typically arise in HBF, especially in ill-conditioned problems such as the Rosenbrock function \cite{AABR}. Their results demonstrate improved convergence rates compared to the classical HBF method, making this approach particularly valuable in nonconvex and large-scale optimization.

Motivated by these developments and related works \cite{AABR,ACFR,LMV}, in this paper we propose to study the evolution of the dynamical system \eqref{eq:hessian_damped} in the context of strongly quasiconvex functions. Furthermore, we also develop two discrete-time algorithms derived from tailored discretizations of the continuous model:
\begin{itemize}
\item \textit{Heavy Ball Method with Hessian Correction}:
\begin{equation}\label{eq:heavy_ball_corrected}
 \begin{split}
  y_k & = x_k + \alpha_{k} (x_k - x_{k-1}) - \theta_k \bigl( \nabla h(x_k) - \nabla h(x_{k-1}) \bigr), \\[1mm]
  x_{k+1} & = y_k - \beta_k \nabla h(x_k),
 \end{split}
\end{equation}
where the correction term approximates the Hessian-driven damping by finite differences.
    
\item \textit{Nesterov-type Method with Adaptive Momentum}:
\begin{equation}\label{eq:nesterov_adapted}
 \begin{split}
  y_k & = x_k + \alpha_{k} (x_k - x_{k-1}) - \theta_k \bigl( \nabla h(x_k) - \nabla h(x_{k-1}) \bigr),\\[1mm]
  x_{k+1} & = y_k - \beta_k \nabla h(y_k),
  \end{split}
 \end{equation}
which extends Nesterov's method with curvature-aware updates.
\end{itemize}

Our analysis bridges continuous-time dynamics and discrete optimization algorithms, establishing exponential and linear convergence rates and illustrating how Hessian-based damping mitigates oscillations more effectively than traditional momentum me\-thods. This work advances the understanding of second-order methods in the context of nonconvex functions by showing a class of ge\-ne\-ra\-li\-zed convex functions for which exponential and linear convergence rates can be obtained and also underscores the potential of combining geometric insights with algorithmic design.

The paper is organized as follows. In Section \ref{sec:02}, we introduce the preliminary definitions and basic results used throughout the paper. In Section \ref{sec:03}, we present a detailed study of the dynamical behavior of second-order system \eqref{eq:hessian_damped}; we discuss the assumptions required to establish convergence rates and their relationship with strongly quasiconvex functions. Furthermore, in Section \ref{sec:04}, we analyze the convergence rates of the discrete‐time algorithms derived from \eqref{eq:hessian_damped}, namely the Heavy Ball method with Hessian correction \eqref{eq:heavy_ball_corrected} and the Nesterov‐type method with adaptive momentum \eqref{eq:nesterov_adapted}, and compare these new results with the state of the art on this subject. Moreover, numerical experiments are given in Section \ref{sec:05} in order to illustrate our new contributions. Finally, future research lines are discussed in Section \ref{sec:05}.

\section{Preliminaries}\label{sec:02}

The inner product in $\mathbb{R}^{n}$ and the Euclidean norm are denoted by $\langle \cdot,\cdot \rangle$ and $\lVert \cdot \rVert$, respectively. The sets $[0, + \infty[$ and $]0, + \infty[$ are denoted by $\mathbb{R}_{+}$ and by $\mathbb{R}_{++}$, res\-pec\-ti\-ve\-ly. Given any $x, y \in \mathbb{R}^{n}$ and any $\beta \in \mathbb{R}$, the following relation holds:
\begin{align}
 \lVert \beta x + (1-\beta) y \rVert^{2} &= \beta \lVert x \rVert^{2} + (1 - \beta) \lVert y\rVert^{2} - \beta(1 - \beta) \lVert x - y \rVert^{2}. \label{iden:1} 
\end{align}

Given any extended-valued function $h: \mathbb{R}^{n} \rightarrow \overline{\mathbb{R}}$, the effective domain of $h$ is defined by ${\rm dom}\,h := \{x \in \mathbb{R}^{n}: h(x) < + \infty \}$. It is said that $h$ is proper if ${\rm dom}\,h$ is nonempty and $h(x) > - \infty$ for all $x \in \mathbb{R}^{n}$. The notion of properness is important when dealing with minimization pro\-blems.

We denote by ${\rm epi}\,h := \{(x,t) \in \mathbb{R}^{n} \times \mathbb{R}: h(x) \leq t\}$ the epigraph of $h$, by $S_{\lambda} (h) := \{x \in \mathbb{R}^{n}: h(x) \leq \lambda\}$ the sublevel set of $h$ at the height $\lambda \in \mathbb{R}$ and by ${\rm argmin}_{\mathbb{R}^{n}} h$ the set of all points of minimum of $h$. A function $h$ is lower se\-mi\-continuous (lsc henceforth) at $\overline{x} \in \mathbb{R}^{n}$ if for any sequence $\{x_k\}_{k} \subset \mathbb{R}^{n}$ with $x_k \rightarrow \overline{x}$, we have $h(\overline{x}) \leq \liminf_{k \rightarrow + \infty} h(x_k)$. Furthermore, the usual con\-ven\-tion  $\sup_{\emptyset} h := - \infty$ and $\inf_{\emptyset} h := + \infty$ is adopted.

A proper function $h$ with convex domain is said to be:
\begin{itemize}
 \item[$(a)$] convex if, given any $x, y \in \mathrm{dom}\,h$, then
 \begin{equation}\label{def:convex}
  h(\lambda x + (1-\lambda)y) \leq \lambda h(x) + (1 - \lambda) h(y),
  ~ \forall ~ \lambda \in [0, 1],
 \end{equation}

 \item[$(b)$] strongly convex on ${\rm dom}\,h$ with modulus $\gamma
 \in \, ]0, + \infty[$ if for all $x, y \in \mathrm{dom}\,h$ and all
 $\lambda \in[0, 1]$, we have
 \begin{equation}\label{strong:convex}
  h(\lambda y + (1-\lambda)x) \leq \lambda h(y) + (1-\lambda) h(x) -
  \lambda (1 - \lambda) \frac{\gamma}{2} \lVert x - y \rVert^{2},
 \end{equation}

 \item[$(c)$] quasiconvex if, given any $x, y \in \mathrm{dom}\,h$, then
 \begin{equation}\label{def:qcx}
  h(\lambda x + (1-\lambda) y) \leq \max \{h(x), h(y)\}, ~ \forall ~
  \lambda \in [0, 1],
 \end{equation}

 \item[$(d)$] strongly quasiconvex on ${\rm dom}\,h$ with modulus
 $\gamma \in \, ]0, + \infty[$ if for all $x, y \in \mathrm{dom}\,h$ and all 
 $\lambda \in[0, 1]$, we have
 \begin{equation}\label{strong:quasiconvex}
  h(\lambda y + (1-\lambda)x) \leq \max \{h(y), h(x)\} - \lambda(1 -
  \lambda) \frac{\gamma}{2} \lVert x - y \rVert^{2}.
 \end{equation}
 \noindent It is said that $h$ is strictly convex (resp. strictly quasiconvex) 
 if the inequa\-li\-ty in \eqref{def:convex} (resp. \eqref{def:qcx}) is strict 
 whenever $x \neq y$ and $\lambda \in \, ]0, 1[$.
 \end{itemize}
 The relationship between all these notions is summarized below (we denote quasiconvex by qcx):
 \begin{align*}
  \begin{array}{ccccccc}
  {\rm strongly ~ convex} & \Longrightarrow & {\rm strictly ~ convex} & \Longrightarrow & {\rm convex} \notag \\
  \Downarrow & \, & \Downarrow & \, & \Downarrow \notag \\
  {\rm strongly ~ qcx} & \Longrightarrow & {\rm strictly ~ qcx} & \Longrightarrow & {\rm qcx}.
  \end{array}
 \end{align*}
None of the reverse implications holds in general. For instance, the Euclidean norm $h_{1}(x) = \lVert x \rVert$ is strongly quasiconvex without  being strongly convex on any bounded convex set (see \cite[Theorem 2]{J-2}) and the function $h_{2} (x) = \frac{x}{1 + \lvert x \rvert}$ is  strictly quasiconvex without being strongly quasiconvex on $\mathbb{R}$ while the other counterexamples are well-known (see \cite{CM-Book,HKS}).

We recall the following existence result for lsc strongly quasiconvex functions.

\begin{lemma}\label{exist:unique} {\rm (\cite[Corollary 3]{Lara-9})}
 Let $K \subseteq \mathbb{R}^{n}$ be a closed and convex set and $h:
 \mathbb{R}^{n} \rightarrow \overline{\mathbb{R}}$ be a proper, lsc, and
 strongly qua\-si\-con\-vex function on $K \subseteq {\rm dom}\,h$ with
 modulus $\gamma> 0$. Then, ${\rm argmin}_{K} h$ is a singleton.
\end{lemma}
Furthermore, if $h$ is lsc and strongly quasiconvex with modulus $\gamma > 0$ on a closed convex set $K \subset \mathbb{R}^{n}$, then the unique minimizer $\overline{x} = {\rm argmin}_{K}\,h$ satisfies that (see \cite[Theorem 5.1]{NS})
\begin{equation}\label{strong min}    h(\overline{x}) + \frac{\gamma}{4} \|x - \overline{x}\|^2 \leq h(x), ~ \forall ~ x \in K.
\end{equation}

Before continuing, let us show some new examples of strongly quasiconvex functions which are not convex.

\begin{remark}\label{rem:exam}
 \begin{itemize}
  \item[$(i)$] Let $h: \mathbb{R}^{n} \rightarrow \mathbb{R}$ be given by $h(x) = \lVert x \rVert^{\alpha}$, with $0 < \alpha < 1$. Clearly, $h$ is nonconvex, but it is strongly quasiconvex on any convex and bounded set by \cite[Corollary 3.9]{NS}. 

  \item[$(ii)$] Let $A, B \in \mathbb{R}^{n\times n}$ be two symmetric matrices, $a, b \in \mathbb{R}^{n}$, $\alpha, \beta \in \mathbb{R}$, and $h: \mathbb{R}^{n} \rightarrow \mathbb{R}$ be the function given by:
  \begin{equation}
   h(x) = \frac{f(x)}{g(x)} = \frac{\frac{1}{2} \langle Ax, x \rangle + \langle a, x \rangle + \alpha}{\frac{1}{2} \langle Bx, x \rangle + \langle b, x 
   \rangle + \beta}.
  \end{equation}
  Take $0 < m < M$ and define:
  $$K := \{x \in \mathbb{R}^{n}: ~ m \leq g(x) \leq M\}.$$ 
   If $A$ is a positive definite matrix and at least one of the following conditions holds:
  \begin{enumerate}
   \item[$(a)$] $B = 0$ (the null matrix),
 
   \item[$(b)$] $h$ is nonnegative on $K$ and $B$ is negative semidefinite,
 
   \item[$(c)$] $h$ is nonpositive on $K$ and $B$ is positive semidefinite,
 \end{enumerate}
  then $h$ is strongly quasiconvex on $K$ with modulus $\gamma = \frac{\lambda_{\min} (A)}{M}$ by \cite[Corollary 4.1]{ILMY}, where $\lambda_{\min} (A)$ is the minimum eigenvalue of $A$.

  \item[$(iii)$] Let $A: \mathbb{R}^{n} \rightarrow \mathbb{R}^{n}$ be a linear operator and $h$ a strongly quasiconvex function with modulus $\gamma \geq 0$. Then $f := h \circ A$ is strongly quasiconvex with modulus $\gamma \sigma_{\min} (A) \geq 0$ (see \cite[Proposition 11$(b)$]{GLM-Survey}), where $\sigma_{\min} (A) \geq 0$ is the minimum singular value of $A$.
  
  \item[$(iv)$] Let $h_{1}, h_{2}: \mathbb{R}^{n} \rightarrow \mathbb{R}$ be two strongly quasiconvex functions with modulus $\gamma_{1}, \gamma_{2} > 0$, respectively. Then $h := \max\{h_{1}, h_{2}\}$ is 
  strongly quasiconvex with modulus $\gamma := \min\{\gamma_{1}, \gamma_{2}\} > 0$ (see \cite{VNC-2}).
 \end{itemize}
\end{remark}

A function $h: \mathbb{R}^{n} \rightarrow \mathbb{R}$ its said to be
$L$-smooth on $K \subseteq \mathbb{R}^{n}$ if it is differentiable on
$K$ and
\begin{equation}\label{L:smooth}
 \lVert \nabla h(x) - \nabla h(y) \rVert \leq L \lVert x - y \rVert, ~
 \forall ~ x, y \in K.
\end{equation}

For $L$-smooth functions, a fundamental result is the descent lemma
\cite[Lemma 1.2.3]{Nesterov-libro}, that is, if $h$ is a $L$-smooth
function on a convex set $K$ with value $L \geq 0$, then for every
$x, y \in K$, we have
\begin{equation}\label{descent:lemma}
 h(y) \leq h(x) + \langle \nabla h(x), y - x \rangle + \frac{L}{2} \lVert
 x - y \rVert^{2}.
\end{equation}
%

In the strongly quasiconvex case, we have (see \cite{J-1}).

\begin{lemma}\label{char:gradient} {\rm (\cite[Theorems 1 and 6]{VNC-2})}
 Let $K \subseteq \mathbb{R}^{n}$ be a convex open set and $h: K \rightarrow \mathbb{R}$ be differentiable function. Then $h$ is strongly
 quasiconvex if and only if there exists $\gamma > 0$ such that for every $x, y \in K$, we have
 \begin{equation}\label{gen:char}
  h(x) \leq h(y) ~ \Longrightarrow ~ \langle \nabla h(y), x - y \rangle
  \leq -\frac{\gamma}{2} \lVert y - x \rVert^{2}.
 \end{equation}
\end{lemma}

Following \cite{L,P1}, we also recall that a differentiable function $h:
\mathbb{R}^{n} \rightarrow \mathbb{R}$ satisfies the
Polyak-{\L}ojasiewicz (PL henceforth) property if there exists $\mu > 0$ such 
that
\begin{equation}\label{PL:def}
 \lVert \nabla h(x) \rVert^{2} \geq \mu (h(x) - h(\overline{x})), ~ \forall
 ~ x \in \mathbb{R}^{n},
\end{equation}
where $\overline{x} \in {\rm argmin}_{\mathbb{R}^{n}}\,h$. The (PL) 
property is implied by strong convexity, but it allows for multiple minima and does not require any convexity assumption (see \cite{KNS-2016} for instance).

\begin{lemma}\label{prop:PL} {\rm (\cite[Theorem 2]{K-1980})}
 Let $h: \mathbb{R}^{n} \rightarrow \mathbb{R}$ be a strongly quasiconvex with modulus $\gamma > 0$ and differentiable with $L$-Lipschitz continuous gradient ($L>0$). Then the (PL) property holds with modulus $\mu := \frac{
 \gamma^{2}}{2L} > 0$, that is,
 \begin{equation}\label{eq:PL}
  \lVert \nabla h(x) \rVert^{2} \geq \frac{\gamma^{2}}{2L} (h(x) -  h(\overline{x})), ~ \forall ~ x \in  \mathbb{R}^n,
 \end{equation}
 where $\overline{x} = {\rm argmin}_{\mathbb{R}^n}\, h$. 
\end{lemma}

The function $h: \mathbb{R} \rightarrow \mathbb{R}$ given by $h(x)=x^{2} + 2 \sin^{2}x$ is an example of a di\-ffe\-ren\-tia\-ble strongly quasiconvex function (with modulus $\gamma = 1/2$) on $\mathbb{R}$ satisfying the (PL) property and without being convex.

\medskip

Now, let us recall the following class of functions (see \cite[Definition 1]{HSS}). Let $h: \mathbb{R}^{n} \to \mathbb{R}$ be a differentiable function. Then $h$ is said to be (strongly) quasar-convex with respect to $\overline{x} \in {\rm argmin}_{\mathbb{R}^{n}}\,h$ if there exists $\kappa \in \, ]0, 1]$ and $\gamma \geq 0$ such that
\begin{align}\label{s:quasarconvex}
 h(\overline{x}) \geq h(x) + \frac{1}{\kappa} \langle \nabla h(x), \overline{x} - x \rangle + \frac{\gamma}{2} \| \overline{x} - x\|^{2}, ~ \forall ~ x \in \mathbb{R}.
\end{align}
If $\gamma > 0$, we said that $h$ is $(\kappa, \gamma)$-strongly quasar-convex while if $\gamma = 0$ we just said that $h$ is $\kappa$-quasar-convex. 

Quasar-convex functions are a class of generalized convex functions which has been studied deeply during the last years in virtue of its applications in machine learning theory (see \cite{HADR,HSS,LV} and references therein). Furthermore, note that (strongly) quasar-convex functions include (strongly) convex functions and (strongly) star-convex functions (see \cite{NP-2006}).

Finally, we recall that for every $z_1, z_2, \ldots, z_m \in \mathbb{R}^n$, we have
\begin{align}\label{eq:lmm:X}
 \|z_1+ z_2 +\ldots +z_m\|^2 \leq m (\|z_1\|^2 + \|z_2\|^2 + \ldots + \|z_m\|^2).
\end{align} 

For a further study on generalized convexity and strongly quasiconvex functions we refer to \cite{CM-Book,GLM-Survey,HKS,HADR,HSS,Lara-9,VNC-2} and references therein.

\section{Second-Order Continuous Dynamical System with a Hessian-Driven Damping}\label{sec:03}

In this section, our goal is to investigate the problem \eqref{min.h} using a dynamical system approach. To that end, we examine the inertial dynamical system with Hessian-driven damping \eqref{eq:hessian_damped}, which we recall below:
\begin{align}
\left\{
 \begin{array}{ll}\label{eq:Dryds}
  \ddot{x}(t) + \alpha\dot{x}(t) + \beta \nabla^2 h (x(t)) \dot{x}(t) + \nabla h (x(t)) = 0, \, t > 0, \\  [2mm]
  x(0) = x_0, \quad \dot{x}(0) = v_0,
 \end{array}
 \right.
\end{align}
\noindent
where $\alpha>0$ and $\beta\geq 0$ are parameters controlling the damping and the influence of the Hessian term. This system dynamics allows us to analyze the behavior of the solution trajectory in the context of the minimization problem \eqref{min.h}.

In order to analyze the dynamical behavior of \eqref{eq:Dryds}, we will assume that there exists $\kappa \in \, ]0, +\infty[$ such that for every  trajectory $x(t)$ of  \eqref{eq:Dryds} we have
 \begin{equation}\label{weak:quasiconvex}
 \langle \nabla h(x(t)), x(t) - \overline{x} \rangle \geq \kappa (h(x(t)) - h(\overline{x})),
 \end{equation}
 where $\overline{x} = {\rm argmin}_{\mathbb{R}^n} h$. This assumption holds trivially with $\kappa =1$ for convex and strongly convex functions, and can be considered as a generalized convexity assumption. Furthermore, assumption \eqref {weak:quasiconvex} was used in \cite{CEG}, and later in \cite{AC2017,ADR2} with a stronger requirement that $\kappa \ge 1$, while it coincides with $\kappa$-quasar-convex (see relation \eqref{s:quasarconvex} with $\gamma = 0$) when $\kappa \in \, ]0, 1]$.

We will analyze sufficient conditions for differentiable strongly quasiconvex functions to satisfy relation \eqref{weak:quasiconvex} in Subsection \ref{subsec:3-2}.

\subsection{Exponential Convergence}\label{subsec:3-1}

The following result establishes exponential convergence of every trajectory to the unique solution of the minimization problem for strongly quasiconvex functions.

\begin{theorem}\label{theo:damping}
 Let $h: \mathbb{R}^n \rightarrow \mathbb{R}$ be a twice differentiable and strongly quasiconvex function with modulus $\gamma > 0$, $\overline{x} = {\rm argmin}_{\mathbb{R}^n} h$, $h^{*} := h(\overline{x})$ and $x(t)$ be a solution of \eqref{eq:Dryds}. Suppose that assump\-tion \eqref{weak:quasiconvex} holds. Then for every $\alpha  \in \, \left] 0, \sqrt{\frac{\gamma (\kappa + 2)^{2}}{8 \kappa}}\right]$ and every {\color{black} $\beta \in \, ]0, \frac{\kappa + 2}{
 \kappa \alpha}]$}, it holds
\begin{align}\label{eq:con.rate}
 \frac{\gamma}{4} \|x(t) - \overline{x}\|^2 \leq  h(x(t)) - h^{*} \leq \mathcal{C} e^{- \frac{\lambda \kappa}{2} t},
\end{align}
where $\lambda:= \frac{2\alpha}{\kappa + 2}$  and $$\mathcal{C}:= h(x_0) - h^{*} + \frac{1}{2} \| \lambda (x_0 - \overline{x}) + v_0 + \beta \nabla h(x_0)\|^2.$$
Consequently, the trajectory $x(t)$ converges exponentially to the unique solution $\overline{x}$ of problem \eqref{min.h} and the function values $h(x(t))$ converge exponentially to the optimal value $h^{*}$. 
\end{theorem}

\begin{proof}
 We define a Lyapunov function $\mathcal{E}: [0, + \infty[ \rightarrow 
 \mathbb{R}_{+}$ as
 \begin{align}\label{eq:energyF}
  \mathcal{E}(t) := h(x(t)) - h^{*} + \frac{1}{2} 
  \| \lambda (x(t) - \overline{x}) + \dot{x}(t) + \beta \nabla h(x(t))\|^2.
 \end{align}
Set 
$$v(t) := \lambda (x(t) - \overline{x}) + \dot{x}(t) + \beta \nabla 
h(x(t)).$$ 
\noindent
Note that 
\begin{align}
 \dot{v}(t)= \lambda \dot{x}(t)+\ddot{x}(t)+ \beta \nabla^2 \,h (x(t)) 
 \, \dot{x}(t) = (\lambda - \alpha) \dot{x}(t) - \nabla h(x(t)).\notag
\end{align}

Differentiating the energy function in \eqref{eq:energyF} and combining it 
with the two above equations, we have
\begin{align}
 \dot{\mathcal{E}} (t) = & ~ \langle  \nabla h(x(t)), \dot{x} (t) \rangle 
 + \langle v(t), \dot{v} (t) \rangle  \notag \\
 = & ~ \langle  \nabla h(x(t)), \dot{x} (t) \rangle + \langle \lambda 
 (x(t) - \overline{x}) + \dot{x}(t) + \beta \nabla h(x(t)), 
 (\lambda - \alpha) \dot{x}(t) - \nabla h(x(t)) \rangle \notag \\
  = & ~ \lambda(\lambda - \alpha)\langle x(t) - 
 \overline{x}, \dot{x}(t)\rangle + (\lambda - \alpha) \|\dot{x}(t)\|^2 
 - \beta \|\nabla h(x(t))\|^2 \notag \\
 & ~ + \beta(\lambda -\alpha)\langle \nabla h(x(t)), \dot{x}(t)\rangle 
 - \lambda \langle x(t) - \overline{x}, \nabla h(x(t))\rangle. \label{eq:M1}
\end{align}
On the other hand, since for every $t> 0$, $h(x(t)) \geq h(\overline{x})$, it follows from Lemma \ref{char:gradient} and Assumption \eqref{weak:quasiconvex} that
\begin{align}\label{inq:key}
 \langle \nabla h(x(t)), x (t) - \overline{x} \rangle \geq \frac{\gamma}{4} \|x(t) - \overline{x}\|^2 + \frac{\kappa}{2} (h(x(t)) -h^*).  
\end{align}
Hence, combining \eqref{inq:key} with \eqref{eq:M1}, and then using \eqref{eq:energyF}, we obtain
\begin{align}
 \dot{\mathcal{E}} (t) \leq & ~ \lambda(\lambda - \alpha)\langle x(t) - \overline{x}, \dot{x}(t)\rangle + (\lambda 
 - \alpha) \|\dot{x}(t)\|^2 - \beta \|\nabla h(x(t))\|^2 \notag \\
 & + \beta(\lambda -\alpha)\langle \nabla h(x(t)), \dot{x}(t)\rangle 
 - \frac{\lambda \gamma}{4} \|x(t) - \overline{x}\|^2 - \frac{\lambda 
 \kappa}{2} (h(x(t)) -h^*) \notag \\
 = & ~ \lambda(\lambda - \alpha) \langle x(t) - \overline{x}, \dot{x}(t)\rangle + (\lambda 
 - \alpha) \|\dot{x}(t)\|^2 - \beta \|\nabla h(x(t))\|^2 \notag \\
 & + \beta(\lambda -\alpha)\langle \nabla h(x(t)), \dot{x}(t)\rangle 
 - \frac{\lambda\, \gamma}{4} \|x(t) - \overline{x}\|^2 
 - \frac{\lambda \kappa}{2} (\mathcal{E} (t) - \frac{1}{2} \lVert 
 v(t) \rVert^{2}), \notag 
\end{align}
that is,
\begin{align}
 \dot{\mathcal{E}} (t) + & \, \frac{\lambda \kappa}{2} \mathcal{E} (t) 
 \leq  \lambda(\lambda - \alpha) \langle x(t) - \overline{x}, \dot{x}(t)\rangle + (\lambda 
 - \alpha) \|\dot{x}(t)\|^2 - \beta \|\nabla h(x(t))\|^2 \notag \\
 & + \beta(\lambda -\alpha)\langle \nabla h(x(t)), \dot{x}(t)\rangle 
 - \frac{\lambda\, \gamma}{4} \|x(t) - \overline{x}\|^2 
 + \frac{\lambda \kappa}{4} \lVert v(t) \rVert^{2}. \notag
\end{align}
Replacing $v(t)$ in the previous equation, after developing and grouping the similar terms, we obtain
\begin{align}
 \dot{\mathcal{E}} (t) + & \, \frac{\lambda \kappa}{2} \mathcal{E} (t) 
 \leq \left(\lambda(\lambda - \alpha) +  \frac{\lambda^{2} \kappa}{2} 
 \right) \langle x(t) - \overline{x}, \dot{x}(t)\rangle + \left(\lambda 
 - \alpha + \frac{\lambda \kappa}{4} \right) \|\dot{x}(t)\|^2 \notag \\
 & + \beta \left( \frac{\lambda \kappa}{2} + \lambda -\alpha \right) 
 \langle \nabla h(x(t)), \dot{x}(t)\rangle + \left( \frac{\lambda \kappa 
 \beta^{2}}{4} - \beta \right) \|\nabla h(x(t)) \|^2 \notag \\
 & + \left( \frac{\lambda^{3} \kappa}{4} - \frac{\lambda\, \gamma }{4}\right) \|x(t) - \overline{x}\|^2 + \frac{\lambda^{2} 
 \kappa \beta}{2} \langle \nabla h(x(t)), x(t) - \overline{x} \rangle. 
 \label{eq:M2}
\end{align}
Now, since
\begin{align*}
 \frac{\lambda^{2} \kappa \beta}{2} \langle \nabla h(x(t)), x(t) - 
 \overline{x} \rangle \leq \frac{\lambda^{3} \kappa}{4} \lVert x(t) - 
 \overline{x} \rVert^{2} + \frac{\lambda \kappa \beta^{2}}{4} \lVert 
 \nabla h(x(t)) \rVert^{2},
\end{align*}
we obtain from \eqref{eq:M2}
\begin{align}
 \dot{\mathcal{E}} (t) + \frac{\lambda \kappa}{2} \mathcal{E} (t) &  \leq \lambda \left(\lambda - \alpha +  \frac{\lambda  \kappa}{2} 
 \right) \langle x(t) - \overline{x}, \dot{x}(t)\rangle + \left(\lambda 
 - \alpha + \frac{\lambda \kappa}{4} \right) \|\dot{x}(t)\|^2 \notag \\
 & + \beta \left( \lambda -\alpha +  \frac{\lambda \kappa}{2} \right) 
 \langle \nabla h(x(t)), \dot{x}(t)\rangle + \beta \left( \frac{\lambda \kappa 
 \beta }{2} - 1 \right) \|\nabla h(x(t)) \|^2 \notag \\
 & + \frac{\lambda }{2}  \left( \lambda^{2} \kappa- \frac{\gamma }{2} 
 \right) \|x(t) - \overline{x}\|^2. \label{eq:M3}
\end{align}
Since $\lambda = \frac{2 \alpha}{\kappa + 2}$, we have that $\lambda - \alpha +  \frac{\lambda  \kappa}{2}=0$, and $\lambda - \alpha + \frac{\lambda \kappa}{4}= - \frac{\kappa \alpha}{2 (\kappa + 2)} < 0$, hence the first three terms of \eqref{eq:M3} are bounded above by zero. Moreover, the assumptions on $\beta$ and $\alpha$ guarantee that the coefficients of the last two terms are negative, so their sum is also non-positive.

Therefore, from the above analysis we deduce
\begin{equation*}
 \dot{\mathcal{E}} (t) + \frac{\lambda \kappa}{2} \mathcal{E} (t) \leq 0,
\end{equation*}
which immediately implies that
 \begin{equation}\label{eq:attch}
  \mathcal{E}(t) \leq \mathcal{E} (0) \, e^{- \frac{\lambda\kappa}{2} t}. 
 \end{equation}
 Recalling the definition of $\mathcal{E}$, we deduce
 \begin{align*}
  h(x(t)) - h(\overline{x}) \leq \mathcal{E}(0) e^{- \frac{\lambda \kappa}{2} t},  
 \end{align*}
which corresponds to the second inequality in \eqref{eq:con.rate}. The first inequality follows directly from \eqref{strong min}. 

Finally, it follows from \eqref{eq:con.rate} that  $x(t)$ converges exponentially to the unique solution $\overline{x}$, and  the function values $h(x(t))$  also converge exponentially to the optimal value $h^{*}$. 
\end{proof}

\begin{remark}
 \begin{itemize}
  \item[$(i)$] As far as we know, Theorem \ref{theo:damping} is the first second-order di\-ffe\-ren\-tial equation with Hessian driven damping for studying differentiable strongly quasiconvex functions. However, in the particular case when $\alpha =\beta = 0$ (i.e., without viscous damping and Hessian terms), particular results may be found in \cite{PK} and \cite{LMV}.

  \item[$(ii)$] The case where $\beta=0$ (i.e., without the Hessian driven-damping term) falls in the Heavy Ball dynamical system considered in \cite{LMV}. Our convergence rate obtained here is comparable to that in \cite{LMV} on constant factors (see \cite[Theorem 23 and Remark 24]{LMV}).
 \end{itemize} 
\end{remark}

As a consequence of Theorem \ref{theo:damping}, we have the following convergence rates.

\begin{corollary}\label{cor.convR}
 Under the assumptions of Theorem \ref{theo:damping}, the following asymptotic exponential convergence rate, as $t \to + \infty$, holds
 \begin{align}
  & h(x(t)) - h^{*} = \mathcal{O} \left(e^{-\frac{\lambda \kappa}{2} t} \right)  ~~ \textit{and}~~\|x(t) - \overline{x}\|^{2} = \mathcal{O} \left(e^{-\frac{\lambda \kappa}{2} t}
  \right).\label{eq:rate01}
 \end{align} 

\noindent Moreover

\medskip
\centerline{
\noindent\begin{tabular}{|c|c|c|c|} 
 \hline 
 $\int_{0}^{t} e^{\lambda s} \| \nabla h(x(s)) \|^2 ds \leq \tilde{\mathcal{C}}$ 
 & $\int_{0}^{t} e^{\lambda s} \| \dot{x}(s) \|^2 ds \leq \tilde{\mathcal{C}}$
  & $\kappa \in (2, +\infty)$ \vspace{0.1cm} \\ \hline
 $e^{-\lambda t} \int_{0}^{t} e^{\lambda s} \| \nabla h(x(s)) \|^2 ds \leq \tilde{\mathcal{C}} e^{- \frac{\lambda \kappa}{2} t}$ & $e^{- \lambda t} \int_{0}^{t} e^{\lambda s} \| \dot{x}(s) \|^2 ds   \leq \tilde{\mathcal{C}} e^{-\frac{\lambda \kappa}{2} t}$ & $\kappa \in [1, 2)$ \vspace{0.1cm} 
 \\ \hline
 $ e^{-\lambda \kappa t} \int_{0}^{t} e^{\lambda \kappa s} \| \nabla h(x(s)) \|^2 ds \leq \tilde{\mathcal{C}} e^{-\frac{\lambda \kappa}{2} t} $ & $ e^{-\lambda \kappa t} \int_{0}^{t} e^{\lambda \kappa s} \| \dot{x}(s)) \|^2 ds \leq \tilde{\mathcal{C}} e^{- \frac{\lambda \kappa}{2} t} $ & $\kappa\in (0,1)$, \vspace{0.1cm}  \\ \hline
\end{tabular}
}
\medskip
\noindent where $\lambda := \frac{2\alpha}{\kappa + 2}$, $\tilde{\mathcal{C}} > 0$ and $h^{*} = h(\overline{x})$. 

\end{corollary}

\begin{proof}
The assertion of \eqref{eq:rate01} follows directly from \eqref{eq:con.rate}.

 From \eqref{eq:energyF} and \eqref{eq:attch} we have that there exists $\mathcal{C}>0$ such that
\begin{equation}
    \| \lambda(x(t) - \overline{x}) + \dot{x}(t) + \beta \nabla h (x(t)) \|^2 \leq   2 \mathcal{C} e^{-\frac{\lambda \kappa}{2} t}.\notag
\end{equation}
Developing the left-hand side term in the above inequality, we have
\begin{align}\label{eq:TM}
    \lambda^2 \|x(t) - \overline{x} \|^2 &+ \| \dot{x}(t)\|^2 + \beta^2 \| \nabla h(x(t))\|^2  + 2 \beta \lambda \langle x(t) - \overline{x}, \nabla  h(x(t)) \rangle  \notag \\
    &+ 2 \langle \dot{x}(t),  \beta \nabla h(x(t)) +   \lambda(x(t) - \overline{x}) \rangle \leq 2 \mathcal{C} e^{-\frac{\lambda \kappa}{2} t}.
\end{align}
Moreover, by \eqref{weak:quasiconvex} we have
\begin{equation}
 \langle x(t) - \overline{x}, \nabla h(x(t)) \rangle \geq \kappa(h(x(t)) - h(\overline{x})). \notag
\end{equation}
Also
\begin{equation}
 \langle \dot{x}(t), \beta \nabla h(x(t)) +  \lambda(x(t) - \overline{x}) \rangle = \frac{d}{dt} \Big( \beta (h(x(t)) - h(\overline{x})) + \frac{\lambda }{2}\| x(t) - \overline{x} \|^2 \Big). \notag
\end{equation}
Combining these two relations with \eqref{eq:TM} and using $\| \dot{x}(t)\|^2 
\geq 0$, we obtain
\begin{align}\label{eq:TM1}
    \lambda \Big( 2 \beta \kappa  ( h(x(t) - h(\overline{x}) ) + \lambda \| x(t) - \overline{x} \|^2 \Big) + \beta^2 \| \nabla h(x(t))\|^2 & \notag\\
    + \frac{d}{dt} \Big( 2 \beta ( h(x(t)) - h(\overline{x}) ) + \lambda \| x(t) - \overline{x} \|^2 \Big) &\leq 2\mathcal{C} e^{-\frac{\lambda\, \kappa}{2} t}.
\end{align}
Define 
\begin{align}
\phi(t) := 2\beta (h(x(t)) - h(\overline{x})) + \lambda \|x(t) - \overline{x}\|^2.\notag
\end{align}
We separate in two cases:

(i) Suppose that $0<\kappa<1$: For this regime, we observe that $\kappa \lambda \|x(t) - \overline{x}\|^2 \leq  \lambda \|x(t) - \overline{x}\|^2$. Combining this inequality with \eqref{eq:TM1} and recalling the definition of $\phi$, we have
\begin{align}
 \frac{d}{dt} \phi(t) + \kappa \lambda  \phi(t) + \beta^2 \|\nabla h(x(t))\|^2 \leq 2C e^{-\frac{\lambda k}{2} t}.\notag
\end{align}
Multiplying by $e^{\kappa \lambda t}$ to both sides of the inequality, we have
\begin{align}
 \frac{d}{dt} (e^{\kappa \lambda t} \phi(t)) + \beta^2 e^{\kappa \lambda t} \|\nabla h(x(t)) \|^2 \leq 2 C e^{\frac{\lambda k}{2} t}. \notag
\end{align}
Integrating, we obtain
\begin{align}
 e^{\kappa \lambda t} \phi(t) - \phi(0)+ \beta^2 \int_{0}^{t} e^{\kappa \lambda s} \|\nabla h(x(s))\|^2  ds \leq \frac{4C}{\kappa \lambda} (e^{\frac{\lambda k}{2} t} -1). \notag
\end{align}
Since $\phi(t) \geq 0$, we have 
\begin{align*}
 e^{-\lambda \kappa t} \beta^{2} \int_{0}^{t} e^{\lambda \kappa s} \| 
 \nabla h(x(s)) \|^2 ds & \leq \frac{4 \mathcal{C}}{\kappa \lambda} 
 e^{-\frac{ \lambda k}{2} t} + \phi(0)e^{- \kappa \lambda t} \\ 
 & = \left( \frac{4 \mathcal{C}}{\kappa \lambda} + \phi(0) e^{- \frac{\lambda k}{2} t} \right) e^{-\frac{\lambda k}{2} t}.
\end{align*}
 (ii) Suppose that $\kappa\geq 1$: Note that $\beta ( h(x(t) - h(\overline{x})) \leq  \kappa \beta ( h(x(t) - h(\overline{x}))$. Then \eqref{eq:TM1} reduces to,
\begin{align}
 \frac{d}{dt} \phi(t) + \lambda \phi(t) + \beta^2 \|\nabla h(x(t)) \|^2 \leq 2 C e^{-\frac{\lambda k}{2} t}.\notag
\end{align}
Integrating,
\[
e^{ \lambda t} \phi(t) - \phi(0)+ \beta^2 \int_{0}^{t}e^{\lambda s} \|\nabla h(x(s))\|^2  ds \leq 2 \mathcal{C} \int_{0}^{t} (e^{(1-\frac{\kappa}{2})\lambda s} ds = \frac{2 \mathcal{C}}{1 - \frac{\kappa}{2}} (e^{(1 - \frac{\kappa}{2})\lambda t} -1).
\]
If $\kappa\in [1, 2)$, then
\begin{align*}
 e^{-\lambda t} \beta^{2} \int_{0}^{t} e^{\lambda s} \| \nabla h(x(s)) \|^2 ds & \leq 
 \frac{4 \mathcal{C}}{2 - \kappa}e^{- \frac{\kappa \lambda}{2} t} + \phi(0) 
 e^{- \lambda t} \\ 
 & = \left( \frac{4 \mathcal{C}}{2 - \kappa} + \phi(0)e ^{- \frac{2-\kappa}{2} \lambda t} \right) e^{- \frac{\kappa \lambda}{2} t}.
\end{align*}
If $\kappa>2$, then
\begin{align*}
 e^{-\lambda t} \beta^{2} \int_{0}^{t} e^{\lambda s} \| \nabla h(x(s)) \|^2 ds \leq 
 \left( \frac{4 \mathcal{C}}{\kappa -2} + \phi(0) \right) e^{-\lambda t}.
\end{align*}
To estimate \(\|\dot{x}(t)\|\), we basically follow the same steps as before, but this time we drop the nonnegative term $\beta^2 \| \dot{x}(t)\|^2$ from \eqref{eq:TM}. 
\end{proof}

\begin{remark}
 The relations in Corollary \ref{cor.convR} indicate that \(\|\nabla h(x(t))\|^2\) and \(\|\dot{x}(t)\|^2\) decay to zero at an exponential rate on average. For strongly convex functions, a similar convergence rate was established in \cite[Theorem 7]{ACFR} considering a specific damping term equal to $2 \sqrt{\gamma}$, where $\gamma$ is the strong convexity parameter. In contrast, Theorem \ref{theo:damping} and Corollary \ref{cor.convR} allow a free damping term, providing greater flexibility in its selection.
\end{remark}

\subsection{Sufficient Conditions}\label{subsec:3-2}

Theorem \ref{theo:damping} requires differentiable strongly quasiconvex functions which satisfy relation \eqref{weak:quasiconvex} ($\kappa>0$), which holds in particular when the function is quasar-convex ($\kappa \in \, ]0,1]$).

First, note that if the function $h: \mathbb{R}^{n} \rightarrow \mathbb{R}$ is
differentiable with $L$-Lipschitz gradient and strongly quasiconvex with
modulus $\gamma> 0$, then it follows from \eqref{gen:char},
\eqref{descent:lemma} and $\nabla h(\overline{x}) = 0$ that
\[
\langle \nabla h(x), x - \overline{x} \rangle \geq \frac{\gamma}{2} \lVert x -
\overline{x} \rVert^{2} \geq \frac{\gamma}{L} (h(x) - h(\overline{x})),
\]
which implies \eqref{weak:quasiconvex} with $\kappa= \frac{\gamma}{L}$.

Second, and as noted in \cite{LV}, differentiable strongly quasi-convex
functions are quasar-convex on bounded convex sets by \cite[Theorem 5.2]{LV},
but this inclusion does not hold on unbounded sets even in one dimension, as
the following example shows.

\begin{example}
\label{rem:compare}Take $h_{1},h_{2}:\mathbb{R}\rightarrow \mathbb{R}$ given by
$h_{2}(x)=x^{2}$ and by
\[
h_{1}(x)=\left \{
 \begin{array}[c]{ll}
 3ex^{2}-2e\left \vert x\right \vert ^{3},~~ & \mathrm{if} ~~ \lvert x \rvert 
 < 1, \\
 e^{n}(1+(e-1)(3(\left \vert x\right \vert -n)^{2}-2(\lvert x \rvert
 -n)^{3})), & n\leq \lvert x \rvert <n+1, ~ \forall~n \in \mathbb{N}.
 \end{array}
 \right.
\]
Note that $h_{1}$ is even, $h_{1}\left(  n\right)  =e^{n}$ and $h_{1}^{\prime
}\left(  n\right)  =0$ for $n\in \mathbb{N}$, $h_{1}\left(  0\right)
=h_{1}^{\prime}\left(  0\right)  =0$, and $h_{1}$ is increasing on
$\mathbb{R}_{+}$.

Then we define $h:\mathbb{R}\rightarrow \mathbb{R}$ as follows:
\[
h(x)=h_{1}(x)+h_{2}(x).
\]
Since $h_{2}$ is strongly convex (in particular, strongly quasiconvex) with
modulus $\gamma_{2}=2$ (see \eqref{strong:convex}) and $h_{1}$ is nondecreasing on
$\mathbb{R}_{+}$, the function $h=h_{1}+h_{2}$ is strongly quasiconvex with
modulus $\gamma_{2}=2$. Indeed, let $\left \vert x\right \vert \leq \left \vert y \right \vert $. Then we have $\left \vert (1-\lambda)x + \lambda y\right \vert \leq \left \vert y\right \vert $ for all $\lambda \in \lbrack0,1]$ and thus
$h_{1}((1-\lambda)x+\lambda y)=h_{1}\left(  \left \vert (1-\lambda)x+\lambda
y\right \vert \right)  \leq h_{1}(\left \vert y\right \vert )=h_{1}\left(
y\right)  $ for all $\lambda \in [0,1]$. Hence,
\begin{align}
h((1-\lambda)x+\lambda y)  &  =h_{1}((1-\lambda)x+\lambda y)+h_{2}%
((1-\lambda)x+\lambda y)\nonumber \\
&  \leq h_{1}(y)+h_{2}(y)-\lambda(1-\lambda)\frac{2}{2}(y-x)^{2}%
\label{eq:ce}\\
&  =h(y)-\lambda(1-\lambda)(y-x)^{2}\nonumber \\
&  =\max \left \{  h(x),h(y)\right \}  -\lambda(1-\lambda)(y-x)^{2}, ~ \forall~\lambda \in [0, 1],\nonumber
\end{align}
Therefore, $h$ is strongly quasiconvex with modulus $\gamma_{2} = 2$.

On the other hand, observe that $h$ is not even quasar-convex for any choice
of the parameters $\kappa \in \,]0,1]$ and $\gamma \geq0$ (see \cite{HSS}).
Indeed, since $h$ is di\-ffe\-ren\-tia\-ble, $h(0)=0$, it follows from 
\eqref{s:quasarconvex} that if $h$ were (strongly) quasar-convex, we would 
have (for $x^{\ast}=0 \in \mathrm{argmin}_{\mathbb{R}}\,h$) that
\[
\frac{1}{\kappa} h^{\prime}(x) \geq \frac{h(x)}{x}+\frac{\gamma}{2}%
x,~~\forall~x>0.
\]
Taking $x=n\in \mathbb{N}$, we have $h^{\prime}(n)=h_{1}^{\prime}(n)=2n$ and
$h(n)=e^{n}+n^{2}>e^{n}$, thus
\[
\frac{1}{\kappa}2\geq \frac{e^{n}}{n^{2}}+\frac{\gamma}{2},~\forall
~n\in \mathbb{N},
\]
which is impossible even for $\gamma=0$. Therefore, $h$ is not quasar-convex.
\end{example}

A sufficient condition for differentiable strongly quasiconvex functions for being quasar-convex on unbounded domains, i.e., for strongly quasiconvex functions to satisfy relation \eqref{weak:quasiconvex}, is given below.

\begin{proposition}
\label{the:prop} Let $K$ be a closed and convex set in $\mathbb{R}^{n}$, $h$ be a differentiable strongly quasiconvex function with modulus $\gamma>0$ and $\overline{x}$ its unique minimizer. If
\begin{align}
 \limsup_{x \in K,\, \lVert x\rVert \rightarrow + \infty} \frac{h(x)}{\lVert x \rVert^{2}} & < + \infty, \label{not:hypercoer} \\
 \limsup_{x \in K,\, \lVert x - \bar{x} \rVert \rightarrow 0} \frac{h(x) - h(\bar{x})}{\lVert x-\bar{x} \rVert^{2}} & < +\infty, \label{not:hypercoer'}
\end{align}
then there exists $\kappa \in \,]0,1]$ such that
\begin{equation}\label{wanted:prop}
 \kappa(h(y) - h(\overline{x})) \leq \langle \nabla h(y), y - \overline{x} \rangle, ~ \forall ~ y \in K, 
\end{equation}
that is, $h$ is quasar-convex on $K$ with modulus $\kappa \in \,]0,1]$ (i.e., satisfies relation \eqref{weak:quasiconvex}).
\end{proposition}

\begin{proof}
Suppose to the contrary that $h$ does not satisfy \eqref{wanted:prop}. Then for every $\kappa \in \,]0,1]$, there exists $y_{\kappa}\in K$ such that
\[
\kappa(h(y_{\kappa})-h(\overline{x}))>\langle \nabla h(y_{\kappa}),y_{\kappa
}-\overline{x}\rangle.
\]
For every $\ell \in \mathbb{N}$ ($\ell \geq 2)$), we take $\kappa = \frac{2}{\ell} \in \, ]0,1]$. Then there exists $y_{\ell} \in K$ such that
\[
\frac{2}{\ell} (h(y_{\ell}) - h(\overline{x})) > \langle \nabla h(y_{\ell}), y_{\ell} - \overline{x} \rangle.
\]
Since $h$ is strongly quasiconvex with modulus $\gamma>0$, it follows from
Lemma \ref{char:gradient} that
\begin{align}
&  \frac{2}{\ell}(h(y_{\ell})-h(\overline{x})) > \langle \nabla h(y_{\ell}), y_{\ell} - \overline{x} \rangle \geq \frac{\gamma}{2} \lVert y_{\ell} - \overline{x}\rVert^{2} \nonumber \\
&  \Longrightarrow \,h(y_{\ell})-h(\overline{x})>\frac{\gamma \ell}{4}\lVert
y_{\ell} -\overline{x}\rVert^{2}\nonumber \\
&  \Longleftrightarrow \, \frac{h(y_{\ell})-h(\overline{x})}{\lVert y_{\ell} - \overline{x} \rVert^{2}} >\frac{\gamma \ell}{4}, ~ \forall~\ell \in \mathbb{N} \nonumber \\
&  \Longrightarrow \, \lim_{\ell \rightarrow+\infty} \frac{h(y_{\ell}) - 
h(\overline{x})}{\lVert y_{\ell} - \overline{x} \rVert^{2}} = + \infty.
 \label{for:conclu}
\end{align}
If $\lVert y_{\ell}\rVert \rightarrow+\infty$ for some subsequence, then
$h(y_{\ell})\rightarrow+\infty$, too, and relation \eqref{for:conclu}
contradicts assumption \eqref{not:hypercoer}. Therefore, $\{y_{\ell}\}_{\ell}$
is bounded and there exists a subsequence, that we denote again by $\{y_{\ell
}\}_{\ell}$, such that $y_{\ell}\rightarrow \widehat{y}\in K$. Using
\eqref{for:conclu} and since $h$ is continuous, we conclude that $\widehat
{y}=\overline{x}$. But then \eqref{for:conclu} again contradicts
\eqref{not:hypercoer'}. Therefore, \eqref{wanted:prop} holds for some
$\kappa \in \,]0,1]$.
\end{proof}

\begin{corollary}
Let $K$ be a closed and convex set in $\mathbb{R}^{n}$, $h$ be a
differentiable strongly quasiconvex function with modulus $\gamma>0$ and
$\overline{x}\in \mathrm{int}\,K$ be its unique minimizer. If $h$ is twice
differentiable at $\bar{x}$ and \eqref{not:hypercoer} holds, then $h$ is 
quasar-convex (i.e., satisfies relation \eqref{weak:quasiconvex}).
\end{corollary}

\begin{proof}
We only need to prove that relation \eqref{not:hypercoer'} holds. Indeed, let
\[
\beta:= \max \left \{  \left \langle \nabla^{2}h\left(  \bar{x}\right)
v,v\right \rangle :\left \Vert v\right \Vert =1\right \}  .
\]
By Taylor's formula, using $\nabla h\left(  \bar{x}\right)  =0$, we obtain for
$y$ close to $\bar{x}$ that
\[
h\left(  y\right)  -h\left(  \bar{x}\right)  =\left \langle \nabla^{2}h\left(
\bar{x}\right)  \left(  y-\bar{x}\right)  ,y-\bar{x}\right \rangle +o\left(
\left \Vert y-\bar{x}\right \Vert ^{2}\right)  \leq \beta \left \Vert y-\bar
{x}\right \Vert ^{2}+o\left(  \left \Vert y-\bar{x}\right \Vert ^{2}\right)
\]
where%
\[
\lim_{y\rightarrow \bar{x}}\frac{o\left(  \left \Vert y-\bar{x}\right \Vert
^{2}\right)  }{\left \Vert y-\bar{x}\right \Vert ^{2}}=0.
\]
Hence,
\[
\limsup_{y\rightarrow \bar{x}}\frac{h\left(  y\right)  -h\left(  \bar
{x}\right)  }{\left \Vert y-\bar{x}\right \Vert ^{2}}\leq \beta.
\]
Then, \eqref{not:hypercoer'} holds. Therefore, the result follows by applying
Proposition \ref{the:prop}.
\end{proof}

Proposition \ref{the:prop} is not true if assumption \eqref{not:hypercoer} does not hold (see Example \ref{rem:compare} where $h$ is twice differentiable at $\bar{x}=0$). It is also not true if the assumption that $h$ is twice differentiable at $\bar{x}$ does not hold, as the next example shows.

\begin{example}
We define first the function $g(x) = \frac{1 - \cos (\pi x)}{2}$, $0\leq x\leq1$. 
This function is differentiable, increasing on $[0, 1]$, and has the properties:
$g(0)=0$, $g(1)=1$, $g^{\prime}(0)=g^{\prime}(1)=0$. Then, we define 
the function $f: \mathbb{R}\rightarrow \mathbb{R}$ by%
\[
 f(x) = \left \{
 \begin{array}[c]{cc}%
 0 & x=0. \\
 \left(  \frac{1}{n^{\frac{3}{2}}} - \frac{1}{\left(  n+1\right)  ^{\frac{3}{2}}}
 \right)  g\left(  n\left(  n+1\right)  \left \vert x\right \vert -n\right)
 + \frac{1}{(n+1)^{\frac{3}{2}}} & \frac{1}{n+1} < \lvert x \rvert \leq 
 \frac{1}{n}, \\
 1 & 1 < \lvert x \rvert.
\end{array}
\right.
\]
Note that $f$ has the following properties: it is even, increasing on $[0,+\infty
\lbrack$ thus quasiconvex, differentiable, $f(\frac{1}{n})=\frac{1}{n^{3/2}}$,
$f^{\prime}(\frac{1}{n})=0$ for all $n\in \mathbb{N}$, $f^{\prime}(0)=0$. Then
we define $h:\mathbb{R}\rightarrow \mathbb{R}$ by
\[
 h(x)=x^{2}+f(x).
\]
The function $h$ is differentiable and satisfies \eqref{not:hypercoer}. It
also strongly quasiconvex; indeed, since $f$ is quasiconvex and $x^{2}$ is
strongly quasiconvex with modulus $2$, for each $x,y\in \mathbb{R}$ with, say,
$\lvert x\rvert \leq \lvert y\rvert$ and $z=\lambda x+(1-\lambda)y$, we observe
that $f(z)\leq f(y)$ and $z^{2}\leq y^{2}-\lambda(1-\lambda)\lVert
y-x\rVert^{2}$, so $h(z)\leq h(y)-\lambda(1-\lambda)\lVert y-x\rVert^{2}$.

However, $h$ does not satisfy \eqref{wanted:prop} for any choice of $\kappa>
0$. In fact, assuming that \eqref{wanted:prop} holds for some $\kappa$ leads
to
\begin{gather*}
\forall~ n \in \mathbb{N}, \quad \kappa \left(  \frac{1}{n^{2}} + \frac
{1}{n^{3/2}} \right)  \leq \frac{2}{n}\frac{1}{n} ~ \Longrightarrow~
\kappa \left(  1 + n^{1/2} \right)  \leq2.
\end{gather*}
This clearly cannot hold for every $n \in \mathbb{N}$. Therefore, $h$ is not quasar-convex.
\end{example}

\section{Hessian Driven Methods}\label{sec:04}

In this section, we prove that temporal discretizations of \eqref{eq:Dryds} 
yield accelerated gradient-type methods, namely the Heavy Ball and Nesterov 
acceleration me\-thods, where the exponential convergence of the continuous
system translates into linear convergence (in both cases).

\subsection{Heavy Ball Acceleration}


Let us consider the Heavy Ball acceleration method with Hessian correction \eqref{eq:heavy_ball_corrected} and constant step sizes, that is,
\begin{align}
\left\{
 \begin{array}{ll}\label{inertialHess:alg}
  ~~~\, y_{k} = x_{k} +   \alpha (x_{k} - x_{k - 1}) - \theta \left(\nabla h(x_k) - \nabla  h(x_{k-1})\right), \\ [2mm]
  x_{k+1} = y_{k} - \beta \nabla h(x_{k}),
 \end{array}
 \right.
\end{align}
where $\alpha \in [0,1]$, $\theta \geq 0$, and $\beta >0$.

In the following theorem, we provide a linear convergence rate for the sequence generated by \eqref{inertialHess:alg} to the optimal solution and the optimal value of pro\-blem \eqref{min.h}.

\begin{theorem}\label{convergenceHeavyball}
 Let $h: \mathbb{R}^n \rightarrow \mathbb{R}$ be a differentiable strongly quasiconvex function with modulus $\gamma > 0$ with $L$-Lipschitz continuous gradient and $\overline{x} \in {\rm argmin}_{\mathbb{R}^n} h$. Suppose that
 \begin{equation}\label{param:hb}
  \alpha \in \,  \left[0, \frac{\sqrt{2}}{2}\right[, ~~ \theta \in \, \left[0, \frac{\alpha}{L \sqrt{3}} \right[, ~~ \beta > \theta (\sqrt{2}-1), ~~ \theta + \beta  \in \, \left]0, \frac{1 - 2 \alpha^{2}}{L} \right].
 \end{equation} 
 Then the se\-quen\-ce $\{x_k\}_{k}$, generated by the Heavy Ball method
 \eqref{inertialHess:alg}, converges linearly to $\overline{x}$ and the sequence $\{h(x_k)\}_{k}$ converges linearly to the optimal value $h^{*} = 
 h(\overline{x})$.
\end{theorem}

\begin{proof}
 First note that \eqref{inertialHess:alg} can be-rewriten as
 \begin{align}\label{eq:T0}
  x_{k+1}- x_{k}= \alpha(x_{k} - x_{k-1}) - (\theta + \beta)\nabla h(x_{k}) + \theta \nabla h(x_{k-1}).
 \end{align}
  Hence,
 \begin{align}\label{eq:T1}
  \|x_{k+1}- x_{k}\|^2 = & \, \alpha^2\|x_{k} - x_{k-1}\|^2 + (\theta + \beta)^2 \|\nabla h(x_{k})\|^2 + \theta^2\|\nabla h(x_{k-1})\|^2 \notag \\
  & - 2 \alpha (\theta + \beta) \langle \nabla h(x_{k}), x_{k} - x_{k-1} \rangle + 2 \alpha \theta \langle \nabla h(x_{k-1}), x_{k} - x_{k-1} \rangle \notag \\
  & - 2\theta(\theta + \beta)\langle \nabla h(x_k), \nabla h(x_{k-1})\rangle.
 \end{align}

 Since $\nabla h$ is Lipschitz continuous, we have
 \begin{align}\label{descentproperty}
  h(x_{k+1}) & \leq h(x_k) + \langle \nabla h(x_k), x_{k+1} -x_k \rangle + \frac{L}{2} \|x_{k+1} - x_k\|^2.
 \end{align}
 Replacing  \eqref{eq:T0} and \eqref{eq:T1} into 
 \eqref{descentproperty}, we obtain
\begin{align}\label{eq:T3}
 & h(x_{k+1}) \le h(x_k) + \frac{\alpha^2 L}{2}\|x_{k} - x_{k-1}\|^2  - (\theta  +\beta) \left(1- \frac{L}{2} (\theta + \beta) \right)\| 
 \nabla h(x_k)\|^{2} \notag \\
 & + \frac{\theta^2 L}{2}\|\nabla h(x_{k-1})\|^2 + \theta \left( 1 - L (\theta + \beta) \right) \langle \nabla h(x_k), \nabla h(x_{k-1})\rangle 
 \notag \\
 & +  \alpha \left(1 - (\theta + \beta)L\right) \langle \nabla h(x_k), x_{k} - x_{k-1} \rangle +  L \alpha \theta \langle \nabla h(x_{k-1}), x_{k} - x_{k-1} \rangle.
\end{align}

Multiplying both sides of the equality \eqref{eq:T1} by $\frac{1 - (\theta + \beta) L}{2 (\theta +\beta)}$ and then adding the results with \eqref{eq:T3}, we obtain
\begin{align}
 h(x_{k+1}) & - h^* + \frac{1 - L (\theta + \beta)}{2(\theta + \beta)} \lVert x_{k+1} - x_{k} \rVert^2 + \frac{\theta + \beta}{2} \lVert \nabla h(x_k) \rVert^2 \notag \\
 & \leq h(x_{k}) - h^* + \frac{\alpha^2}{2 (\theta + \beta)} \lVert x_{k} 
 - x_{k-1} \rVert^2 + \frac{\theta^2}{2 (\theta + \beta)} \lVert \nabla h(x_{k-1}) \rVert^2 \notag \\
 & + \frac{\alpha \theta}{\theta + \beta} \langle \nabla h(x_{k-1}), x_{k} - x_{k-1} \rangle.\notag
\end{align}
Since
$$2 \frac{\alpha \theta}{\theta + \beta} \langle  \nabla h(x_{k-1}), x_{k} 
 - x_{k-1} \rangle \leq \frac{\alpha^2}{\theta + \beta}\|x_{k} - x_{k-1}\|^2 
 + \frac{\theta^2}{\theta + \beta} \|\nabla h(x_{k-1})\|^2,$$
the last inequality becomes
\begin{align}
 h(x_{k+1}) - h^* \leq & ~ h(x_{k}) - h^* + \frac{\alpha^2}{(\theta + \beta)} \lVert x_{k} - x_{k-1} \rVert^2 + \frac{\theta^2}{(\theta + \beta)} \lVert \nabla h(x_{k-1}) \rVert^2 \notag \\
 & - \frac{1 - L (\theta + \beta)}{2 (\theta + \beta)} \lVert x_{k+1} - x_{k} \rVert^2 - \frac{\theta + \beta}{2} \lVert \nabla h(x_k) \rVert^2. \label{eq:T4}
\end{align}
Defining the energy function
\begin{align}\label{energy:hb}
\mathcal{E}_{k} := h(x_{k}) - h^* +\frac{\alpha^2}{(\theta + \beta)} \lVert x_{k} - x_{k-1} \rVert^2 + \frac{\theta^2}{(\theta + \beta)} \lVert \nabla h(x_{k-1}) \rVert^2,    
\end{align}
we can write the last inequality as
\begin{align}
 \mathcal{E}_{k+1} - \mathcal{E}_k & \leq - \left( \frac{\theta + \beta}{2} - \frac{\theta^2}{\theta + \beta}\right)\|\nabla h(x_{k})\|^2 - \frac{ \left(1-(\theta + \beta)L - 2\alpha^2 \right)}{ 2(\theta + \beta)} \|x_{k+1} - x_k\|^2 \notag \\ 
 & \leq - \rho (\|\nabla h(x_{k})\|^2 + \|x_{k+1} -x_k\|^2), \label{energyEstimate}
\end{align}
where 
\begin{equation}\label{def:rho01}
 \rho := \min\left\{ \frac{\theta + \beta}{2} - \frac{\theta^2}{\theta + \beta}, \frac{1}{2 (\theta + \beta)} \left(1-(\theta + \beta)L - 2 \alpha^2 \right) \right\},
\end{equation} 
and $\rho > 0$, by \eqref{param:hb}.

It remains to bound $\|\nabla h(x_{k})\|^2 + \|x_{k+1} - x_k\|^2$ below by $\mathcal{E}_{k}$. Indeed, it fo\-llows from Lemma \ref{prop:PL} that 
\begin{equation} \label{firstEst}
 h(x_k) - h^* \leq  \frac{2L}{\gamma^2}\|\nabla h(x_k)\|^2.
\end{equation}

Using \eqref{eq:lmm:X}, \eqref{eq:T0} and the Lipschitz continuity of $\nabla h$, we have
\begin{align}
 \alpha^2\|x_{k} - x_{k-1}\|^2 & = \|x_{k+1} - x_k + \theta(\nabla h(x_k) 
 - \nabla h(x_{k-1}) ) + \beta \nabla h(x_k)\|^2 \notag \\
 & \leq 3 \|x_{k+1} - x_{k}\|^2 + 3 \theta^2 L^2 \|x_{k} - x_{k-1}\|^2 + 
 3 \beta^2 \|\nabla h(x_{k})\|^2, \notag
\end{align}
which implies
\begin{align}\label{eq:T5}
 \alpha^2\|x_{k} - x_{k-1}\|^2 \leq \frac{3 \alpha^2}{\alpha^2 - 3 \theta^2 
 L^2} \|x_{k+1} - x_{k}\|^2 + \frac{3 \alpha^2 \beta^2}{\alpha^2 - 3 
 \theta^2 L^2} \| \nabla h(x_{k})\|^2.
\end{align}
Moreover, it follows from \eqref{eq:T0} that
\begin{align} 
 \theta^2 \| \nabla h(x_{k - 1}) \|^2 & = \| x_{k+1} - x_k - \alpha(x_{k} -
  x_{k-1}) + (\theta + \beta) \nabla h(x_k) \|^2 \notag \\
  & \leq 3 \|x_{k+1} - x_{k}\|^2 + 3 \alpha^2 \|x_{k} - x_{k-1}\|^2 + 3 
 (\theta + \beta)^2 \| \nabla h(x_{k}) \|^2, \notag 
\end{align}
which in turn, combined with \eqref{eq:T5} yields
\begin{align}
 \theta^2 \| \nabla h(x_{k - 1}) \|^2 \leq & ~ 3 \left(1 + \frac{3 \alpha^2}{
 \alpha^2 - 3 \theta^2 L^2} \right) \| x_{k + 1} - x_{k} \|^2 \notag \\
 & + 3 \left( (\theta + \beta)^2 + \frac{3 \alpha^2 \beta^2}{\alpha^2 - 3 
 \theta^2 L^2} \right) \| \nabla h(x_{k})\|^2. \label{eq:T6}
\end{align}
Hence, replacing \eqref{firstEst}, \eqref{eq:T5} and \eqref{eq:T6} into \eqref{energy:hb}, we obtain
\begin{align}
 \mathcal{E}_{k} &= h(x_{k}) - h^*  +\frac{1}{\theta +\beta} \left(
 \alpha^2 \|x_{k} - x_{k-1}\|^2 + \theta^2\| \nabla h(x_{k-1})\|^2\right)
 \notag \\
 & \leq  \frac{1}{\theta +\beta} \left( 3 + \frac{12 \alpha^2}{\alpha^2 -
 3 \theta^2 L^2} \right) \lVert x_{k+1} - x_{k} \rVert^2 \notag \\
 & ~~~~ + \left(\frac{2L}{\gamma^2} + 3(\theta + \beta) + \frac{12 \alpha^2
 \beta^2}{(\alpha^2 - 3\theta^2 L^2)(\theta + \beta)} \right) \|\nabla
 h(x_{k})\|^2 \notag \\
 & \leq \sigma \left( \|\nabla h(x_k)\|^2 +  \|x_{k+1} - x_{k} \|^2 \right),
 \notag
\end{align}
where
\begin{align}\label{def:sigma01}
\sigma := \max\left\{ \frac{1}{\theta +\beta} \left( 3 + \frac{12 \alpha^2}{\alpha^2 - 3 \theta^2 L^2} \right), \frac{2L}{\gamma^2} + 3 (\theta + \beta) + \frac{12 \alpha^2 \beta^2}{(\alpha^2 - 3\theta^2 L^2) (\theta + \beta)} \right\},
\end{align}
and $\sigma > 0$ by \eqref{param:hb}. 

\noindent
Finally, combining the last inequality with \eqref{energyEstimate}
we deduce 
$$\mathcal{E}_{k+1} \leq \left (1 - \frac{\rho}{\sigma} \right) \mathcal{E}_k \leq \left(1 - \frac{\rho}{\sigma} \right)^k \mathcal{E}_1,$$
which means that the sequence $\{\mathcal{E}_k\}_{k}$ converges linearly to $0$ (because $\rho < \sigma$). As a consequence,  the sequence $\{h(x_k)\}_{k}$  converge linearly to the optimal value $h^*$ and the sequences $\{\|\nabla h(x_{k})\|\}_k$ and  $\{\|x_{k} - x_{k-1}\|\}_k$  converge linearly to $0$.
 
Finally, since 
\begin{equation}
 \langle \nabla h(x_{k}), x_{k} - \overline{x} \rangle \geq \frac{\gamma}{2} 
 \lVert x_{k} - \overline{x} \rVert^{2} ~ \Longrightarrow ~ \|x_{k} - \overline{x} 
 \| \leq \frac{2}{\gamma} \| \nabla h(x_{k}) \|, \notag
 \end{equation} 
we conclude that $\{x_k\}_{k}$ converges linearly to $\overline{x}$.
\end{proof}
As a direct consequence of the above theorem, we have the following convergence rate result.

\begin{corollary}[Convergence rate]
 Under the assumptions of Theorem \ref{convergenceHeavyball}, we have
 the following convergence rate
 \begin{align}
 &\frac{\gamma}{4} \|x_{k+1} - \overline{x}\|^2 \leq h(x_{k+1}) - h^{*} \leq \left(1 - \frac{\rho}{\sigma} \right)^k \mathcal{E}_{1}, \\
  & \|x_{k+1} - x_{k}\|^2 \leq \frac{\theta + \beta}{\alpha^2} \left(1-\frac{\rho}{\sigma} \right)^k \mathcal{E}_1.
 \end{align}
 Moreover, for all $k \in \mathbb{N}$, we have
 \begin{align*}
  \|\nabla h(x_{k}) \| & \leq \left(1 - \frac{\rho}{\sigma} \right)^{\frac{k-1}{2}}
  \frac{\sqrt{\theta + \beta}}{\theta}  \sqrt{\mathcal{E}_1}, 
 \end{align*}
 where $\mathcal{E}_{1}$, $\rho$ and $\sigma$ are defined in \eqref{energy:hb}, \eqref{def:rho01} and \eqref{def:sigma01}, respectively.
\end{corollary}

\begin{remark}
In the particular case when $\theta=0$ (i.e., without Hessian correction term), Algorithm \ref{inertialHess:alg} reduces to the Heavy Ball gra\-dient method studied in \cite{LMV}, with a slightly different convergence rate (see \cite[Coro\-lla\-ry 26]{LMV}). These convergence rates are consistent with the well-known fact that strongly convex functions exhibit linear convergence (see \cite{Ant,ACFR,P2}).
\end{remark}

\subsection{Nesterov Acceleration }

Let us consider now the Nesterov's accelerated gradient method with Hessian correction \eqref{eq:nesterov_adapted} and constant step sizes, that is,
\begin{align}
\left\{
\begin{array}{ll}\label{NestHess:alg}
 ~~~\, y_{k} = x_{k} +   \alpha (x_{k} - x_{k - 1}) - \theta \left(\nabla h(x_k) - 
 \nabla  h(x_{k-1})\right), \\ [2mm] 
  x_{k+1} = y_{k} - \beta \nabla h(y_{k}),
 \end{array}
 \right.
\end{align}
where $\alpha\in [0,1]$, $\theta\geq 0$, and  $\beta >0$.

Before presenting the main convergence result, we prove an auxiliary result.

\begin{proposition}\label{prop:Q}
 Let $h: \mathbb{R}^n \rightarrow \mathbb{R}$ be a differentiable function with $L$-Lipschitz continuous gradient and strongly quasiconvex with modulus $\gamma > 0$, $\overline{x} = {\rm argmin}_{\mathbb{R}^n} h$, $\eta>1$ 
 and $0 < \beta < \frac{\gamma}{\eta L^2}$.
 Then the sequence generated by \eqref{NestHess:alg} satisfies:
 \begin{equation}\label{eq:001}
  \|x_{k+1} - \overline{x}\|^2 +  \mu_2 \|x_{k+1} - y_{k}\|^2 \leq \mu_1 
  \|y_{k} - \overline{x}\|^2, ~ \forall ~ k \in \mathbb{N},
 \end{equation}
 where 
 \begin{align}\label{eq:mu}
  \mu_1 := \frac{1}{1 + \beta( \gamma - \eta \beta L^2)} \in \, ]0, 1[ ~~ 
  {\rm and} ~~ \mu_2 := \frac{1 - \frac{1}{\eta}}{ 1 + \beta (\gamma - \beta 
  \eta L^2)} >0.
 \end{align}
\end{proposition}

\begin{proof}
 Note that 
 \begin{align}
  & \|y_k - \overline{x}\|^2 - \|x_{k+1} -  \overline{x}\|^2 = \|y_k - x_{k+1}\|^2 
  + 2 \langle y_k - x_{k+1}, x_{k+1} - \overline{x} \rangle \notag \\
  & = \|y_k - x_{k+1}\|^2 + 2 \beta \langle \nabla\, h(y_k) - \nabla\, h(x_{k+1}),
  x_{k+1} - \overline{x} \rangle + 2 \beta \langle \nabla\, h(x_{k+1}), x_{k+1} 
  - \overline{x} \rangle \notag \\
  & \geq \|y_k - x_{k+1}\|^2 - 2 \beta L \|y_k - x_{k+1} \| \|x_{k+1} - 
  \overline{x}\| + \beta \gamma \|x_{k+1} - \overline{x}\|^2 \notag \\
  & \geq \|y_k - x_{k+1} \|^2 - \frac{1}{\eta} \|y_k - x_{k+1} \|^{2} - \eta 
  (\beta L)^2 \| x_{k+1} - \overline{x} \|^{2} + \beta \gamma \|x_{k+1} - 
  \overline{x} \|^2 \notag \\
  & = \left(1- \frac{1}{\eta} \right) \| y_k - x_{k+1}\|^2 + \beta \left(\gamma 
  - \eta \beta L^2 \right) \|x_{k+1} - \overline{x} \|^2. \label{eq:002} 
 \end{align}
Here, we have used Lemma \ref{char:gradient}, the Lipschitz continuity of $\nabla\, h$ and the property: For all $a,b\in \mathbb{R}_+$, it holds $2ab \leq \eta a^2 + \frac{1}{\eta} b^2$ for every $\eta>0$. 

Then, it follows from \eqref{eq:002} that
\begin{equation}
 \left(1 + \beta\left(\gamma - \eta \beta L^2 \right) \right) \|x_{k+1} 
 - \overline{x}\|^2 + \left(1 - \frac{1}{\eta} \right) \|x_{k+1} - y_{k}\|^2 \leq 
 \|y_{k} - \overline{x}\|^2,\notag
\end{equation}
which implies \eqref{eq:001}.
\end{proof}

In the following theorem, we provide a linear convergence rate for the sequence generated by \eqref{NestHess:alg} to the optimal solution and the optimal value of pro\-blem \eqref{min.h}.

\begin{theorem}\label{convergenceNesterov2}
 Let $h: \mathbb{R}^n \rightarrow \mathbb{R}$ be a differentiable with $L$-Lipschitz conti\-nuous gradient and strongly quasiconvex function with modulus $\gamma > 0$ and $\overline{x} = {\rm argmin}_{\mathbb{R}^n} \,h$. Let $\eta >1$, $0<\beta< \frac{\gamma}{\eta L^2}$ and $\mu_1, \mu_2$ be defined as in \eqref{eq:mu}. Let $\epsilon \in \, \left] 0, \frac{1}{\mu_1} -1 \right[$ and $\alpha$, $\theta$ be satisfying
 \begin{equation}\label{condition}
  \alpha + \theta L \le \frac{\mu_1 \mu_2 (1+ \epsilon)}{\mu_1 \mu_2 (1+\epsilon) + \mu_1+ \frac{\mu_1}{\epsilon} + \mu_2}.
 \end{equation}
  Then the sequence 
 $\{x_k\}_{k}$, generated by \eqref{NestHess:alg}, converges linearly to
 $\overline{x}$ and the sequence $\{h(x_k)\}_{k}$ converges linearly to the 
 optimal value $h^{*} = h(\overline{x})$.
\end{theorem}

\begin{proof}
 The main idea is try to bound $\|x_{k+1} - y_k\|$ and $\|y_k -\overline{x}\|$, then replacing into \eqref{eq:001} of Proposition \ref{prop:Q}. Indeed, we have
\begin{align}
 & \|x_{k+1} - y_{k}\|^{2} =\|x_{k+1} - x_k - \alpha (x_k- x_{k-1}) + \theta (\nabla h(x_k) - \nabla h(x_{k-1})\|^2) \notag \\
 &=\|x_{k+1} - x_{k}\|^{2} + \alpha^{2}
 \|x_{k} - x_{k-1}\|^{2} - 2 \alpha \langle x_{k+1} - x_{k}, x_{k} - x_{k-1} \rangle\notag\\
 & + 2 \theta \langle x_{k+1} - x_{k}, \nabla h(x_k) - \nabla h(x_{k-1}) \rangle - 
 2 \alpha \theta \langle x_{k} - x_{k-1},\nabla h(x_k) - \nabla h(x_{k-1}) \rangle 
 \notag \\
 & + \theta^2 \| \nabla h(x_k) - \nabla h(x_{k-1}) \|^2 \notag \\
 & \geq \|x_{k+1} - x_{k}\|^{2} + \alpha^{2} \|x_{k} - x_{k-1}\|^{2} - 
 \alpha \left( \|x_{k+1} - x_{k} \|^2 + \|x_{k} - x_{k-1}\|^2 \right) \notag \\
 & - \theta L \left( \| x_{k+1} - x_{k} \|^2 + \|x_{k} - x_{k-1}\|^2 \right) 
 - \left( \alpha^2 \|x_{k} - x_{k-1}\|^2 + \theta^2 \| \nabla h(x_{k}) - \nabla
 h(x_{k-1}) \|^2\right) \notag \\
 & + \theta^2 \| \nabla h(x_{k}) - \nabla h(x_{k-1}) \|^2  \notag.
 \end{align}
This implies 
\begin{align}\label{eq:d2}
 & \| x_{k+1} - y_{k} \|^{2} \geq (1- (\alpha + \theta L)) \| x_{k+1} - x_k\|^2 
 - (\alpha + \theta L) \| x_{k} - x_{k-1} \|^2.
\end{align}

On the other hand, by the triangle inequality 
\begin{align}
    \|y^k - \overline{x}\| 
    &= \|x_k +\alpha(x_k - x_{k-1}) - \theta (\nabla h(x_k) 
 - \nabla h(x_{k-1})) - \overline{x}\| \notag \notag\\
 & \le \|x_k -\overline{x}\| + (\alpha +\theta L)\|x_k - x_{k-1}\|.\notag
\end{align}
Hence, using Cauchy-Schwartz inequality we can deduce
\begin{align}\label{vu2}
    \|y^k - \overline{x}\| ^2
 & \le (1+\epsilon)\|x_k -\overline{x}\|^2 + \left(1+\frac{1}{\epsilon}\right)(\alpha +\theta L)^2\|x_k - x_{k-1}\|^2.
\end{align}
Combining \eqref{eq:d2} and \eqref{vu2} with \eqref{eq:001}, noting that $(\alpha +\theta L) \in [0, 1[$ we get 
 \begin{align*}
  &\|x_{k+1} - \overline{x}\|^2 +  \mu_2 (1- (\alpha + \theta L)) \| x_{k+1} - x_k\|^2 
 - \mu_2(\alpha + \theta L) \| x_{k} - x_{k-1} \|^2
  \notag \\
 & \leq \mu_1 
(1+\epsilon)\|x_k -\overline{x}\|^2 + \mu_1\left(1+\frac{1}{\epsilon}\right)(\alpha +\theta L)^2\|x_k - x_{k-1}\|^2 \notag \\
&\leq \mu_1 
(1+\epsilon)\|x_k -\overline{x}\|^2 + \mu_1\left(1+\frac{1}{\epsilon}\right)(\alpha +\theta L)\|x_k - x_{k-1}\|^2
 \end{align*}
or equivalently,
 \begin{align}\label{mainEstimate}
  &\|x_{k+1} - \overline{x}\|^2 +  \mu_2 (1- (\alpha + \theta L)) \| x_{k+1} - x_k\|^2 
  \notag \\
  &\leq \mu_1 (1+\epsilon)\|x_k -\overline{x}\|^2 + \left(\mu_1+\frac{\mu_1}{\epsilon} + \mu_2\right)(\alpha +\theta L)\|x_k - x_{k-1}\|^2.
 \end{align}
Denoting for all $k \ge 1$ the Lyapunov function
\begin{align}\label{energy:nes}
 \mathcal{E}_k:= \|x_{k} - \overline{x}\|^2 +  \mu_2 (1- (\alpha + \theta L)) \| x_{k} - x_{k-1}\|^2 \ge 0,
\end{align}
we can deduce from the last inequality that 
\begin{align*}
 \mathcal{E}_{k+1} & \le \mu_{1} (1 + \epsilon) \mathcal{E}_{k} \\ 
 & - \left( \mu_{1} (1 + \epsilon) \mu_{2} (1 -  (\alpha + \theta L)) - \left( \mu_{1} + \frac{\mu_{1}}{\epsilon} + \mu_{2} \right) (\alpha + \theta L) \right) \|x_k - x_{k-1} \|^{2}. 
\end{align*}
Since $\alpha$ and $\theta$ satisfy condition \eqref{condition}
we have 
$$
\mu_1 (1+\epsilon)  \mu_2 (1- (\alpha + \theta L)) - \left(\mu_1+\frac{\mu_1}{\epsilon} + \mu_2\right)(\alpha +\theta L) \ge 0.
$$
Hence, it follows from the last inequality that 
$$ 0 \le \mathcal{E}_{k+1} \leq \mu_1 (1 + \epsilon) \mathcal{E}_k.$$
Since $\mu_1 (1+\epsilon) \in (0,1)$, we conclude that the sequence $\{\mathcal{E}_k\}_{k}$ converges linearly to $0$ with a linear rate $ q:=\mu_1 (1+\epsilon)  \in (0,1)$. As a result, the sequence $\{x_k\}_{k}$ converges linearly to the unique solution $\overline{x}$. The convergence of the value functions is guaranteed by the descent property \eqref{descent:lemma}.
\end{proof}

As a direct consequence of the above theorem, we have the following 
convergence rate result for the Algorithm generated by \eqref{NestHess:alg}.

\begin{corollary}[Convergence rate]
 Under the assumptions of Theorem \ref{convergenceNesterov2} and that $x_0 := x_{-1}$, we have the following convergence rate
  \begin{align*}
   \|x_{k} - \overline{x} \|^2 & \leq \left (\mu_1 (1+\epsilon) \right)^{k-1} \mathcal{E}_1, \\
  \|x_k -  x_{k-1}\|^2 & \le \left( \frac{\mu_1 (1 + \epsilon)}{ \mu_2 (1- (\alpha + \theta L))} \right)^{k-1} \mathcal{E}_1.
 \end{align*}
 Moreover,
 \begin{align}
  & h(x_{k}) - h^{*} \leq \frac{L}{2} \left( \mu_1 (1 + \epsilon) \right)^{k-1} 
  \mathcal{E}_1,
  \end{align}
 where $\mathcal{E}_{1}$ is defined in \eqref{energy:nes} while $\mu_{1}$ and $\mu_{2}$ in \eqref{eq:mu}, respectively.  
\end{corollary}

\begin{remark}
Nesterov accelerated gradient methods with Hessian-driven dam\-ping have been 
well studied in the context of strongly convex functions (see, e.g., \cite{ACFR}). 
For broader classes of nonconvex functions, we refer to \cite{HADR}, where the authors analyze the case of strongly quasar-convex functions (see \cite[Subsection 3.2]{HADR}).
\end{remark}

\section{Numerical Experiments}\label{sec:05}


In this section, we present some numerical examples to illustrate the performance of the Heavy Ball and Nesterov methods with Hessian-driven damping comparing with the vanilla ones.

\begin{example} \label{exam1} 
 The function $h: \mathbb{R} \rightarrow \mathbb{R}$ given by $h(x)=x^{2}+2 \sin^{2}x$ is an example of a strongly quasiconvex function with modulus $\gamma=1/2$ without being convex \cite[Example 30]{LMV}. Clearly, the gradient $\nabla h(x) := 2x + 2 \sin(2x)$ is Lips\-chitz continuous with constant $L=6$. This function has a unique minimizer $x^* =0$. For the Heavy Ball method with Hessian, we select $\alpha = 0.8$, $\theta = 0.05$, and $\alpha = 0.8$, $\theta = 0$ for Heavy Ball method (i.e. without Hessian). For the Nesterov method with Hessian, we select $\alpha = 0.6$, $\theta = 0.05$, and $\alpha = 0.9$, $\theta=0$ for Nesterov method (i.e. without Hessian). All algorithms were run with the same stepsize $\beta = 1/4L$ and $x_0=x_1= 3$.  The performance of four algorithms are displayed in Figure \ref{fig0} indicating the advantage of including the Hessian damping by reducing oscillations. 
  \begin{figure}[htbp]
 \centering 
 \includegraphics[width=0.45\linewidth]{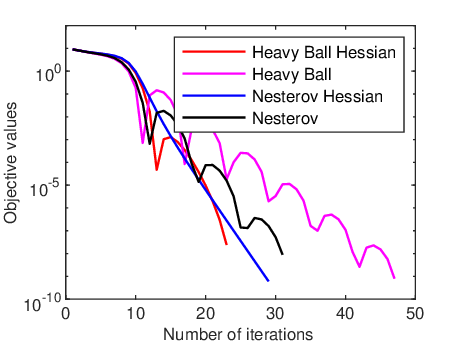} 
 \includegraphics[width=0.45\linewidth]{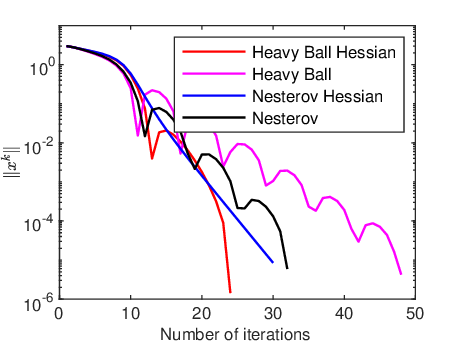} 
 \caption{Comparison in objective values (left) and trajectories (right) for the strongly quasiconvex function $h(x) = x^2 + 3\text{sin}^2(x)$.} \label{fig0}  
\end{figure}
\end{example}
\begin{example} \label{exam2} 
 Let us consider the function $h: \mathbb{R}^{2} \rightarrow \mathbb{R}$ 
given by:
\begin{equation}\label{num:test-1}
 h (x, y) = x^{2} + a y^{2} - \frac{1}{x^{2} + a y^{2} + b} + c,
\end{equation}
for some constant $a,b,c >0$. Clearly, $h$ is twice differentiable and nonconvex (see Figure \ref{fig1} below). Moreover, it can be verified that $h$ it is strongly quasiconvex with modulus and its gradient is Lipschitz (see e.g. \cite[Example 31]{LMV} for a proof of a particular case). 

\begin{figure}[htbp]
\centering
\begin{tikzpicture}
  \begin{axis}[
    view={45}{30},
    domain=-4:4,
    y domain=-4:4,
    samples=60,
    samples y=60,
    xlabel={$x$},
    ylabel={$y$},
    zlabel={$h(x, y)$},
    colormap/viridis,
    colorbar
  ]
    \addplot3[surf] {
      (1/20)*x^2 + (1/5)*y^2 - 1/(x^2 + 4*y^2 + 0.3)
    }
    ;
  \end{axis}
\end{tikzpicture}
\caption{An illustration of the strongly quasiconvex function $h$ defined in 
 \eqref{num:test-1}.} \label{fig1}  
\end{figure}
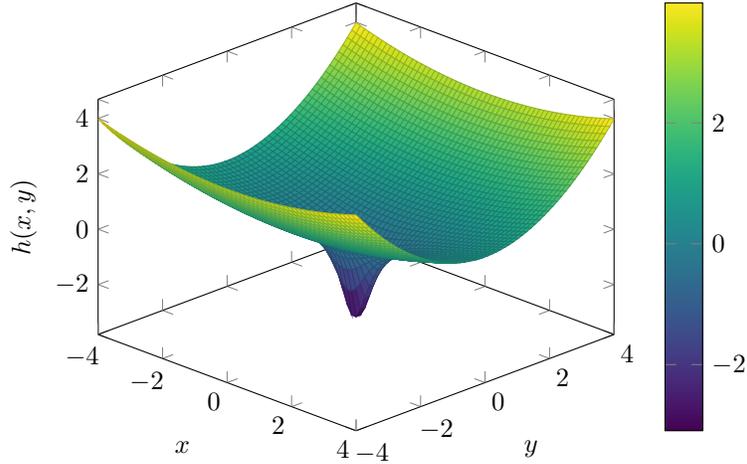
To make the problem ill-conditioned, in the experiment we choose $a=100$, $b=1$ and $c=1$. As the Lipschitz constant is difficult to estimate, we did some parameters turning such that the algorithms work well. For the Heavy Ball method with Hessian, we select $\alpha = 0.8$, $\theta = 0.004$, and $\alpha = 0.8$, $\theta=0$ for Heavy Ball method (i.e. without Hessian). For the Nesterov method with Hessian, we select $\alpha = 0.9$, $\theta = 0.009$, and $\alpha = 0.9$, $\theta=0$ for Nesterov method (i.e. without Hessian). All algorithms were run with the same stepsize $\beta = 0.0025$ and $x_0=x_1=(3,3)$.  The performance of four algorithms are displayed in Figure \ref{fig2} showing again the advantage of including the Hessian damping by reducing oscillations. 

\begin{figure}[htbp]
 \centering 
 \includegraphics[width=0.45\linewidth]{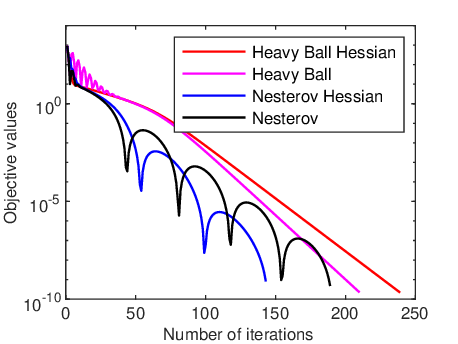} 
 \includegraphics[width=0.45\linewidth]{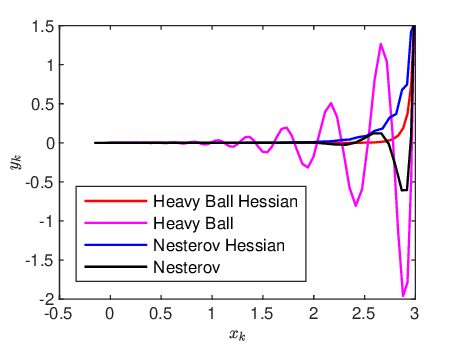} 
 \caption{Comparison in objective values (left) and trajectories (right) for the strongly quasiconvex function $h$ defined in 
 \eqref{num:test-1}.} \label{fig2}  
\end{figure}

\end{example} 

\section{Conclusions}\label{sec:06}
We have further advanced the study of strongly quasi-convex optimization by introducing a second-order dynamical system driven by the Hessian. The discretization of this proposed system naturally recovers two classical momentum-based algorithms with Hessian damping: the Heavy ball method and the Nesterov method. The incorporation of Hessian-driven damping not only acce\-le\-ra\-tes convergence but also reduces the oscillations commonly observed in traditional momentum methods. Future research will explore the behavior of the proposed methods in settings beyond strong quasiconvexity. 

\section{Declarations}







\subsection{Availability of supporting data}

No data sets were generated during the current study. 

\subsection{Author Contributions}

 All authors contributed equally to the study conception, design and implementation and wrote and corrected the manuscript.

\subsection{Competing Interests}

There are no conflicts of interest or competing interests related to this manuscript.

\bigskip

\noindent \textbf{Acknowledgements}
 A part of this paper was completed when Hadjisavvas, Lara an Vuong were visiting the Vietnam Institute for Advanced Study in Mathematics (VIASM), in Hanoi, Vietnam, during March and April 2025. These authors would like to thank VIASM for their support and hospitality. This research was partially supported by ANID--Chile under project Fondecyt Regular 1241040 and by ECOS-ANID project ECOS240031 (Lara), and by BASAL fund FB210005 for center of excellence from ANID-Chile (Marcavillaca).

\end{document}